\documentclass[11pt]{article}
\usepackage{hyperref}
\usepackage{amssymb,amsmath, amsthm}
\usepackage{graphicx}
\usepackage{graphics}
\usepackage{psfrag}

\usepackage[all]{xy}
\usepackage[applemac]{inputenc}

\newtheorem{lemma}{Lemma}[section]

\newtheorem{proposition}{Proposition}[section] 
 
\newtheorem{definition}{Definition}[section]

\theoremstyle{definition}
\newtheorem{remark}{Remark}[section]
\newtheorem{example}{Example}[section]

\newcommand{\Rb}{\mathbb{R}}

 \author{Mário M. Gra\c{c}a  \thanks{Departamento de Matem\'{a}tica,  IDMEC,
LAETA,  Instituto Superior T\'ecnico, Universidade  de Lisboa,   Lisboa, Portugal.}}

\begin{document}
 
\title{Recursive families of higher order iterative maps}
\maketitle

  
  \begin{abstract}
  
  \noindent
To approximate a simple root of an equation we construct families of  iterative maps of higher order of convergence. These maps are based on model functions which can be written as an inner product. The main family of maps discussed  is defined recursively and is called {\it Newton-barycentric}.  We illustrate the application of Newton-barycentric maps in two worked examples, one dealing with a typical least squares problem and the other showing how to   locate simultaneously a great number of extrema of the Ackley's function.

  \end{abstract}



\section{Introduction}\label{introd}
The classical Newton's iterative scheme for approximating the roots of an equation has been generalized by many authors in order to define iterative maps going from cubical to arbitrary orders of convergence (see for instance \cite{traub}, \cite{rheinboldt}, \cite{collatz}, \cite{householder}, \cite{labelle}, \cite{sebah},  \cite{torres}, \cite{weerakoon} and references therein).
The primary aim of this work is to present a systematic construction of families of iterative maps of higher order of convergence.

\medskip
\noindent
All the  iterative maps $t$  to be considered have a common structure  $t(x)=x- \left[\phi(x)  \right]^{-1} f(x)$, where $f$ is a real function with a simple zero at $z$.  The function  $\phi$  will be called a {\it model function}. This model function depends on  $f$ and on another function  $h$ which we name {\it step function}.

\medskip
\noindent
The families of iterative maps to be discussed are constructed by choosing distinct model functions $\phi$. We remark that our approach does not follow the traditional  path for generating iterative maps by direct or inverse hyperosculatory interpolation (see, for instance \cite{traub}) or  Taylor expansions around a zero of $f$ (see for instance \cite{sebah}).
 
 \medskip
\noindent
The paper is organized as follows. In Section~\ref{modelo} we introduce the notion of a model function $\phi$ and prove that the iterative map $t(x)=x- \left[\phi(x)  \right]^{-1} f(x)$ has a certain order of convergence (see Proposition~\ref{prop1}). The proof of Proposition~\ref{prop1}  leads to the definition of  a step function $h$. Then, taking distinct model functions we construct  the so\--called Taylor\--type and  barycentric\--type maps.  In both cases, the map $\phi$ can be written as the Euclidean inner product of two vectorial functions, one depending on  $f$ and the other on the step function $h$.  Such inner product enables us to prove that the iterative map $\phi$ is indeed a  model function and the correspondent iterative map $t$ has a  certain order of convergence (see propositions~\ref{propB} and \ref{lemaC}).

\medskip
\noindent
In the following section we construct recursively families of iterative maps $t_j$ with $j=0,1,\ldots, k$ (for a given positive integer $k$), called Newton\--Taylor and Newton\--barycentric.  The respective step function $h_j$ (defined recursively)  uses the classical Newton's method as a starter. The model function in the Newton\--Taylor family  is of Taylor\--type whereas for Newton\--barycentric family is of barycentric type.  The main result is Proposition~\ref{propD} which shows that each member $t_j$ of the referred families has order of convergence at least $j+2$. We end Section \ref{secrec} by briefly referring how to extend our iterative maps  in order to deal with functions $f$ defined in $\Rb^n$.

\medskip
\noindent
The last section is devoted to the application of some Newton-barycentric formulas to concrete examples. The first example is a typical least squares problem. The other example shows the ability of these maps  to   locate simultaneously  extrema of the Ackley's function \cite{ackley}, a  function widely used   for testing optimization algorithms (see for instance \cite{more}).  This last numerical example shows that higher order iterative maps might be relevant  in those real-world applications where it is important  to get a simultaneous localization of  a great number of extrema in $\Rb^2$, starting from a suitable set of data  points in the domain of the objective function.
\section{Iterative maps derived from model functions}\label{modelo}

Given  a real\--valued  function   $f$ defined on an open set $D\subset \Rb$,  we assume that $f$ is sufficiently smooth  in a neighborhood of a {\it simple zero} $z$ of $f$. In what follows we construct certain families of  iterative maps $t$ generating a sequence $x_{k+1}=t(x_k)$, $k=0,1,\ldots$, converging locally to $z$.

\medskip
\noindent
Let us first recall the notion of {\it order of convergence} of an iterative map $t$ (see for instance \cite{traub}). We say that  $t$ has order order of convergence $p\geq 1$ if there exists a positive constant $c$, such that
$$
\lim_{x \rightarrow z} \displaystyle{  \frac{|z-t(x)|}{| z-x|^p} }=c.
$$
\noindent
 It is well-known that an advantage of the superlinear convergence (that is, when  $p>1$)  is the assurance of the existence of a neighborhood of $z$ where the sequence $x_{k+1}=t(x_k)$ converges to $z$ (see for instance \cite{dennis1}). However, in general, maps of higher order of convergence have expressions of increasing complexity, and so increasing its computational cost. This is one reason why among the  families of iterative maps  obtained in Section \ref{secrec}, the  Newton-barycentric formulas are the ones used in the worked examples.  These formulas are computationally more economic than the  Newton\--Taylor ones, which are only considered here as an illustration of our constructive process of generating iterative methods.

\begin{proposition}\label{prop1}
Let $z$ be a simple zero of a function  $f: D\subset \Rb\mapsto \Rb$ and $\phi$ a sufficiently smooth function in  a neighborhood of 
$z$, such that its derivatives $\phi^{(i)}$
satisfy the  $j+1$ equalities
\begin{equation}\label{eq2}
\begin{array}{l}
 \phi^{(i)}(z)=\displaystyle{\frac{ f^{(i+1)}(z)}{i+1}},\qquad i=0,1,\ldots, j,
\end{array}
\end{equation}
for $j\geq 0$ a fixed integer.
Then, for any initial value $x_0$ sufficiently close to $z$, the iterative process $x_{k+1}=t(x_k)$, $k=0,1,\ldots$,  with
\begin{equation}\label{eq1}
t(x)=x-\phi^{-1} (x)\, f(x),
\end{equation}
 converges to  $z$  and its order of convergence  is at least $j+2$.
\end{proposition}

\begin{proof} 
From \eqref{eq1} it is obvious that the zero $z$ of $f$ is a fixed point of the map $t$ (that is, $t(z)=z$). Let us consider the function $\,\Delta t$ defined  by
$$ 
\Delta t(x)= t(x)-x.
$$
Note that  its derivatives are
$$\Delta^{(1)} t(x)= t^{(1)}(x)-1,\quad \mbox{and}\quad \Delta^{(i)} t(x)= t^{(i)}(x), \quad \mbox{for}\quad  i\geq 2.$$
We now use induction on $j$ to prove that the hypotheses in \eqref{eq2} imply that  $\Delta t(z)=0$, $\Delta^{(1)} t(z)= -1$, $\Delta^{(j)} t(z)= 0$, for $j\geq 2$, and consequently  $t$ has the referred order of convergence.

\medskip
\noindent
Let $j=0$. Rewriting   \eqref{eq1} as
\begin{equation}\label{eq4}
\phi(x)\, \Delta t(x)= - f(x),
\end{equation}
and applying the derivative operator to this equation, we have
\begin{equation}\label{eq5}
\phi^{(1)} (x)\, \Delta t (x)+ \phi(x)\, \Delta ^{(1)}t(x)= - f^{(1)} (x).
\end{equation}
Since  $\Delta t (z)=0$, and $\phi$ satisfies  \eqref{eq2} with $i=0$,   it follows that
$
f^{(1)} (z)\, \Delta^{(1)} (z)=-f^{(1)} (z).
$
As   $z$ is a simple zero for $f$, then 
$$
\Delta^{(1)} t (z)=-1\,\Longleftrightarrow\, t^{(1)}(z)=0,
$$
which means that the iterative process generated by   $t$ has local order of convergence $p$ at least $2$. That is, $p\geq j+2$.

\medskip
\noindent
Let $j=1$. Differentiating  \eqref{eq5}, we get
$$
\phi^{(2)} (x)\, \Delta t (x)+2\,  \phi^{(1)}(x)\, \Delta ^{(1)}t(x)+\phi(x)\, \Delta^{(2)} t(x)= - f^{(2)} (x).
$$
  Since   $\Delta t(z)=0$ and $\Delta^{(1)} t(z)=-1$,   we obtain 
$$
-2\, \displaystyle{ \frac{f^{(2)}(z)}{2} } + f^{(1)} (z)\, \Delta^{(2)} t (z)= - f^{(2)} (z).
$$
Therefore $\Delta^{(2)} t (z)=t^{(2)}(z)=0$, and so the iterative process has local order of convergence at least $3=j+2$.

\medskip
\noindent
For an integer $m\geq 2$, assume that
$$
\phi^{(j)} (z)=\displaystyle{\frac{f^{(j+1)} (z)}{j+1}   }, \quad \mbox{for}\quad j=0,1,\ldots,m,
$$
and
\begin{equation}\label{num1}
\Delta^{(1)} t(z)=-1, \quad  \Delta^{(j)} t(z)=0,  \quad \mbox{for}\quad j=2,3,\ldots,m.
\end{equation}
Let us show that $ \Delta^{(m+1)} t(z)=t^{(m)}(z)=0$.
From \eqref{eq4}  and the Leibniz's rule for the derivatives  of the product, we have
$$
\begin{array}{ll}
& \phi^{(m+1)} (x)\, \Delta t(x)+\binom{m+1}{1}  \phi^{(m)} (x)\, \Delta^{(1)} t(x)+\cdots+ \binom{m+1}{m}  \phi^{(1)} (x)\, \Delta^{(m)} t(x)+\\
&\hspace{2cm}+  \phi^{(0)} (x)\, \Delta^{(m+1)} t(x)= - f^{(m+1)}(x).
\end{array}
$$
Thus, by the induction hypotheses, we obtain
$$
-\binom{m+1}{1} \,\displaystyle{\frac{ f^{(m+1)} (z)}{m+1}}+ f^{(1)}(z)\, \Delta^{(m+1)} t (z)= - f^{(m+1)} (z) \, \Leftrightarrow \Delta^{(m+1)} t (z)=0.
$$
Hence the iterative map $t_{m+1}$ has local order of convergence $p\geq m+2$ and the proof is complete.
\end{proof}
\begin{remark}\label{remnovo}
The well-known result on the local order of convergence of the  Newton's  map $t(x)=x-\left[ f^{(1)}(x)  \right]^{-1} f(x)$ follows immediately from  Proposition \ref{prop1}. It is enough to see that \eqref{eq2} is verified for $j=0$, i.e. $\phi^{(0)}(z)=f^{(1)}(z)$, and so $t$ has local order of convergence at least 2.
\end{remark}

\noindent
A function  like $\Delta t=t(x)-x$, satisfying the properties \eqref{num1} in the proof of Proposition \ref{prop1}, will be called a {\it step function} and a function $\phi$ satisfying  \eqref{eq2}   will be called  a  {\it model} function.

\begin{definition}\label{defstand}
Let $z$ be a simple zero of a function  $f: D\subset \Rb\mapsto \Rb$,  $h$ and $\phi$ sufficiently smooth functions in a neighborhood of $z$, and $j\geq 0$ a fixed  integer.
\begin{itemize}
\item A function $\phi$ is called a model function if it satisfies the $j+1$ conditions \eqref{eq2}.
\item A function $h$  is called a  step function at $x=z$ (or simply a step function) if it satisfies the following $j+1$ equalities: 
 \begin{equation}\label{hs}
\begin{array}{l}
h(z)=0,\quad 
h^{(1)} (z)=-1\quad \mbox{and}\quad
h^{(i)} (z)=0, \quad \mbox{for}\quad i=2,3,\ldots, j.
\end{array}
\end{equation}
\end{itemize}
\end{definition}
\subsection{The Taylor\--type  maps  }\label{subsec1}

As before we assume throughout that $z$ is a simple zero of a real function $f$. Let $k\geq 0$ be an integer. We now construct a family of iterative maps based on the  following function

\begin{equation}\label{eq10B}
\phi_k(x)= f^{(1)} (x)+ \displaystyle{ \frac{f^{(2)}(x)}{2 !}   }\, h(x) +  \displaystyle{ \frac{f^{(3)}(x)}{3 !}   }\, h^2(x) + \cdots+  \displaystyle{ \frac{f^{(k+1)}(x)}{(k+1) !}   }\, h^k(x),
\end{equation}
where  $h$ and $h^k$ denote respectively a given step function and its $k$\--th power.
%
In order to show that the iterative process generated by a map $t_k$, defined by
\begin{equation}\label{eq11B}
t_k(x)= x-  \left[ \phi_k(x)\right] ^{-1}  \, f(x),
\end{equation} 
 has order of convergence at least $k+2$, we write $\phi_k$ as an inner product of two vectorial  functions,  one depending on $f$ and the other on the step function $h$.  This simplifies considerably the necessary  computations because of the orthogonality of a certain basis of $\Rb^{k+1}$ (see Lemma \ref{lemaB} below).

\medskip 
\noindent
The function $\phi_k$ in \eqref{eq10B} can be written as the following Euclidean  inner product
$$
\phi_k(x)=\langle U_k(x), V_k(x)\rangle,
$$
with
\begin{equation}\label{hs1}
V_k(x)=\left(1, h(x), h^2(x), h^3(x), \cdots, h^k(x)\right),
\end{equation}
and 
$$U_k(x) = \left(f^{(1)}(x),  \frac{f^{(2)}(x)}{2 !} , \frac{f^{(3)}(x)}{3 !},\ldots, \frac{f^{(k+1)}(x)}{(k+1) !} \right).$$

\medskip
\noindent  
We now establish some properties of the function $V_k$ which are necessary to  the proof of the  Proposition \ref{propB} below.
 
 \begin{lemma}\label{lemaB} 
Let  $h$ be a  step function at $x=z$  and $M_k$ the $(k+1)\times (k+1)$ matrix whose rows are  the derivatives $V_k^{(0)}(z), V_k^{(1)}(z), \ldots, V_k^{(k)}(z)$ of $V_k$ in  \eqref{hs1}. Then, the set
$$
{\cal V}_k= \left\{ V_k^{(0) }(z), V_k^{(1) }(z), \ldots, V_k^{(k) }(z) \right\},
$$
is  an orthogonal basis of $\Rb^{k+1}$. In particular, the matrix $M_k$ is diagonal and its rows   are
\begin{equation}\label{num2}
\begin{array}{ll}
 V_k^{(0)}(z)&=( 1,0,0,0,\cdots,0   )\\
V_k^{(1)}(z)&=( 0,-1,0,0,\cdots,0   )\\
V_k^{(2)}(z)&=( 0,0,2!,0,\cdots,0   )\\
V_k^{(3)}(z)&=( 0,0,0,- 3!,\cdots,0   )\\
&\vdots\\
V_k^{(k)}(z)&=(0,0,0,0,\cdots, (-1)^k\, k!).
\end{array}
\end{equation}

\end{lemma}
\begin{proof}
Using the definition of step function,  the $k$\--fold differentiation of  $V_k$ gives  that each vector   $V_k^{(i)}(z)$ has only a nonzero component which is  the $(i+1)$\--th one. In particular, this entry is $V^{(i)}_{k,i+1}(z) = (-1)^i\, i!$, for $i=0,1,\ldots, k$. Therefore, the set  ${\cal V}_k $ is obviously orthogonal and  $M_k$ is diagonal.
\end{proof}

\begin{lemma}\label{lemaA}
For a positive integer $m$,  the sum
$$
s_m=1- \displaystyle{ \frac{ 1}{2}  }\binom{m}{1}+ \displaystyle{ \frac{1 }{3}  }\binom{m}{2}- \ldots + (-1)^m \, \displaystyle{ \frac{ 1}{m+1}  }\binom{m}{m}
$$
is
$$
s_m= \displaystyle{ \frac{1}{m+1}  }.
$$
\end{lemma}
\begin{proof}
The proof is straightforward  by induction on $m$ and by well-known properties of the binomial coefficients.
\end{proof}

\begin{proposition}\label{propB}
Let $f$ be a function satisfying the hypotheses of Proposition \ref{prop1},  $h$ a  step function and $\phi_k$  given by   \eqref{eq10B}. 
 Then, $\phi_k$ is a model function and   $t_k= x- \left[\phi_k(x)\right]^{-1} f(x)$ has local order of convergence at least $k+2$.
 \end{proposition}
\begin{proof}  
Once we prove that $\phi_k$ is a model function the result that $t_k$ has local order of convergence at least $k+2$ follows from Proposition~\ref{prop1}.

\medskip
\noindent
Let us use induction on $k$ to prove that $\phi_k$ is a function satisfying \eqref{eq2}. For $k=1$,  the function $\phi_1$ can be written as the inner product
$$
\phi_1(x)=\langle U_1(x), V_1(x)\rangle,
$$
with
$$
\begin{array}{l}
U_1(x)= \left( f^{(1)} (x),  \displaystyle{ \frac{f^{(2)}(x)}{2 !}   }\right)\quad \mbox{and}\quad
V_1(x)= \left( 1, h(x)\right).
\end{array}
$$
Since $h$ is a  step function it satisfies  $h(z)=0$ and $h^{(1)} (z)=-1$, and so
$$
\phi_1(z)= \phi_1^{(0)}(z)= \langle U_1(z), V_1(z)\rangle= f^{(1)}(z).
$$
Now, the first derivative of $\phi_1$ is 
$$
\phi_1^{(1)}(x)=\langle U_1^{(1)}(x), V_1^{(0)}(x)\rangle+ \langle U_1^{(0)}(x), V_1^{(1)}(x)\rangle.
$$
At $x=z$, we have
$$
\begin{array}{l}
U_1^{(0)}(z)= \left( f^{(1)} (z),  f^{(2)}(z)/2\right),\qquad  
U_1^{(1)}(z)= \left( f^{(2)} (z), f^{(3)}(z)/2\right)
\end{array}
$$
and by  Lemma \ref{lemaB},  we obtain  $V_1^{(0)}(z)= ( 1,0)$ and $V_1^{(1)}(z)= (0,-1)$. 
Then,
\begin{align*}\nonumber\label{C2}
\phi_1^{(1)}(z)&=\langle U_1^{(1)}(z), V_1^{(0)}(z)\rangle+\langle U_1^{(0)}(z), V_1^{(1)}(z)\rangle\\
&= f^{(2)}(z)-  \displaystyle{ \frac{f^{(2)}(z)}{2 }   }=  \displaystyle{ \frac{f^{(2)}(z)}{2 }   }.
\end{align*}
Thus, condition \eqref{eq2}  holds for $k=1$.

\medskip
\noindent
For an integer $m\geq  2$, the induction basis  is  $h^{(1)}(z)=-1$,    $h^{(i)}(z)=0$, with $i=0,2,\ldots,m$,  and  $\phi_m^{(0)}(z)=f^{(1)} (z)$,  $\phi_m^{(1)}(z)=f^{(1)} (z)/2$, $\ldots$, $\phi_m^{(m-1)}(z)=f^{(m)} (z)/m$. As
$$
\phi_m(x)=\langle U_m(x),V_m(x)\rangle,
$$
the derivative of  order $m$ of the  inner product gives
\begin{equation}\label{eqesd}
\begin{array}{ll}
\phi_m^{(m)}(x)&= < U_m^{(m)}(x), V_m^{(0) }  (x)   >+ \binom{m}{1}\, < U_m^{(m-1)}(x), V_m^{(1)}  (x)   >+\\
\\
& \hspace{0.5cm}+ \binom{m}{2}\, < U_m^{(m-2)}(x), V_m^{(2)} (x)> +
\cdots+ \binom{m}{m}\, < U_m^{(0)}(x), V_m^{(m)}(x)>. 
\end{array}  
\end{equation}
Let us prove that $\phi_m^{(m)}(z)=f^{(m+1)} (z)/(m+1)$.   We have
$$
\begin{array}{l}
U_m^{(0)}(z)= \left(  f^{(1)} (z),\,  \displaystyle{ \frac{f^{(2)}(z)}{2! }   },\cdots , \displaystyle{ \frac{f^{(m+1)}(z)}{(m+1)! }   }  \right)\\
U_m^{(1)}(z)= \left(  f^{(2)} (z),\,  \displaystyle{ \frac{f^{(3)}(z)}{2! }   },  \cdots,  \displaystyle{ \frac{f^{(m+2)}(z)}{(m+1)! }   }  \right)\\
\hspace{2cm} \vdots\\
U_m^{(m)}(z)= \left(  f^{(m+1)} (z),\,  \displaystyle{ \frac{f^{(m+2)}(z)}{2! }   }, \cdots,  \displaystyle{ \frac{f^{(2\,m+1)}(z)}{(m+1)! }   }  \right),
\end{array}
$$
 and so  by \eqref{num2} the evaluation 
 of \eqref{eqesd} at $x=z$,  gives
$$
\begin{array}{ll}
\phi_m^{(m)}(z)&=f^{(m+1  )} (z)- \displaystyle{ \frac{\binom{m}{1}}{2}   }\, f^{(m+1  )}(z)+\\
&\hspace{0.5cm}+\displaystyle{ \frac{\binom{m}{2}}{3}   }\, f^{(m+1  )}(z)-
\cdots + (-1)^m \displaystyle{ \frac{\binom{m}{m}}{m+1}   }\, f^{(m+1  )}(z)= \displaystyle{ \frac{f^{(m+1)} (z)}{m+1}   } ,
\end{array}
$$
where the last equality follows from Lemma \ref{lemaA}, which completes the proof.
\end{proof}

\medskip
\noindent
A function $\phi_k$ of the form \eqref{eq10B} will be called a Taylor\--type model function and the  map $t_k$ given by \eqref{eq11B} will be referred as a Taylor\--type iterative map.

\medskip
\noindent In paragraph \ref{subbar} we will discuss another family of iterative maps which is computationally more interesting than Taylor-type ones in the sense  it is deduced from a set of model functions $\phi_k$ which  only uses the first derivative of the function $f$.

\medskip
\noindent
The first four Taylor\--type maps are displayed in Table \ref{tab1}. Formulas  similar to those in Table \ref{tab1} have been attributed to Euler and Chebyshev (see Traub \cite{traub}).
    \begin{table}
$$
\hspace{-1cm}
\begin{array}{ | l | }
\hline
t_1(x) = x- \displaystyle{ \frac{2\, f(x)}{2\, f^{(1)}(x)+ f^{(2)}(x)\, h(x)}   }\\
t_2(x) =  x- \displaystyle{ \frac{6\, f(x)^{}}{6\, f^{(1)}(x)+3  f^{(2)}(x)\, h(x)+  f^{(3)}(x)\, h^2(x)}   }\\
t_3(x)= x- \displaystyle{ \frac{24\, f(x)}{24\, f^{(1)}(x)+12  f^{(2)}(x)\, h(x)+4\,  f^{(3)}(x)\, h^2(x)+ f^{(4)}(x)\, h^3(x)}   }\\
t_4(x) = x- \displaystyle{ \frac{120\, f(x)}{120\, f^{(1)}(x)+60  f^{(2)}(x)\, h(x)+20\,  f^{(3)}(x)\, h^2(x)+5\, f^{(4)}(x)\, h^3(x)+ f^{(5)}(x)\, h^4(x)}   }\\
\hline
  \end{array}
  $$
  \caption{ First four Taylor's type maps. \label{tab1}}
  \end{table}

\subsection{The  barycentric\--type  maps}\label{subbar}

\noindent
The Taylor-type iterative maps $t_k$  were constructed using a model function $\phi_k$  defined as an inner product  of two vectorial functions $U_k$ and $V_k$ depending respectively on the first $k$ derivatives of $f$ and on the powers of the step function $h$.
 We now consider another type of iterative maps $t_k$ by modifying the model function $\phi_k$ as follows: $\phi_k$ is the inner product  of a constant vectorial function  $U_k$  and a  function  $V_k$ depending only on the first derivative $f^{(1)}$  evaluated at $x+ i\, h(x)$, for $i=0,\ldots, k$. Notably, we take
\begin{equation}\label{lc0}
U_k(x) = \left( a_0,a_1,a_2,\cdots, a_k  \right) =\mathbf{a},
\end{equation}
and
\begin{equation}\label{lc1}
V_k(x)=\left( f^{(1)}(x),f^{(1)}(x+h(x)),\cdots, f^{(1)}(x+k\, h(x))  \right),
\end{equation}
where $h$ is a step function. 
If one proves that $\phi_k=\langle\mathbf{a}, V_k(x)\rangle$ is a model function
 then, by Proposition~\ref{prop1},  the respective process $t_k(x)= x-\left[\phi_k(x)\right]^{-1} f(x)$ has  order of convergence at least $k+2$. 

\medskip
\noindent
The next proposition shows that $\phi_k$ is a model function if and only if $U_k(x)=\mathbf{a}$ is the unique  solution  of a non homogeneous linear system. Moreover,  this solution represents the barycentric coordinates of   $\phi_k$ in a basis defined by the components of $V_k$.

\begin{proposition}\label{lemaC}
Let $f$  be a function satisfying the hypotheses of Proposition~\ref{prop1},  $h$ a  step function, $k\geq 0$ a fixed integer and
 $\phi_k=\langle U_k, V_k\rangle$, 
  with $U_k$ and $V_k$ defined by \eqref{lc0} and \eqref{lc1}. That is, 
  \begin{equation}\label{lc4}
\phi_k(x)=a_0 f^{(1)} (x)+ a_1 f^{(1)}( x+ h(x)) +\cdots + a_k f^{(1)}( x+k\, h(x)).
\end{equation}
Then, the derivative of order $k$  of $V_k$, evaluated at $x=z$,  is $V^{(k)}_k(z)=D_k \, R_k$,
where $D_k$ and $R_k$ are the following $(k+1)\times (k+1)$ matrices 
$$D_k=\operatorname{diag} \left( f^{(1)}(z), f^{(2)}(z),\cdots, f^{(k+1)}(z)\right),
$$
and
$$
R_k=\begin{bmatrix}
1&1&1&1&\ldots&1\\
1&0&-1&-2&\ldots&-(k-1)\\
1&0&1&2^2&\ldots&(k-1)^2\\
\vdots&\vdots&\vdots&\vdots&\vdots&\vdots\\
1&0&(-1)^k &(-1)^k 2^k&\ldots&(-1)^k (k-1)^k\\
\end{bmatrix}.
$$
Furthermore,
\begin{itemize}
\item[(i)] The function $\phi_k$ is a model function
if and only if $$U_k=\mathbf{a}=\left( a_0,a_1,a_2,\cdots, a_k  \right) $$ is the (unique) solution of the  linear system
\begin{equation}\label{lc5}
R_k\, \mathbf{a} =\mathbf{b},\quad \text{with}\quad \mathbf{b}=\left(1,1/2,1/3,\cdots,1/(k+1)\right).
\end{equation}
Also, this solution satisfies the equality 
\begin{equation}\label{lc8}
\sum_{i=0}^k a_i=1.
\end{equation}
\item[(ii)] If $f^{(i)}(z)\neq 0$ for $i=1, \ldots, k$, then the function $U_k=\mathbf{a}$  represents   the (normalized) barycentric coordinates of the model function  $\phi_k$ relative to the
basis,
$$
{\cal V}_k=\left\{  f^{(1)} (x),  f^{(1)}( x+ h(x)),\cdots,  f^{(1)}( x+k\, h(x))  \right\}.
$$
\end{itemize}
Moreover,  the iterative process generated by $t_k(x)= x-\left[\phi_k(x)\right]^{-1} f(x)$ has   order of convergence at least $k+2$.
\end{proposition}

\begin{proof}
For $i=0,1,\ldots,k$ the  derivatives of order $i$ of  $V_k$,  evaluated at  $x=z$, are:
\begin{align*}\nonumber
V_k^{(0)}& = f^{(1)} (z) \left(1,1,1,1,1,\cdots,1     \right)\\\nonumber
V_k^{(1)}&= f^{(2)} (z) \left(1,0,-1,-2,-3,\cdots,-(k-1)     \right)\\ \nonumber
V_k^{(2)}&= f^{(3)} (z) \left(1,0,1,2^2,3^2,\cdots,(k-1)^2     \right)\\\nonumber
\vdots \\
V_k^{(k)}&= f^{(k+1)} (z) \left(1,0,(-1)^k, (-1)^k 2^k, (-1)^k 3^k, \cdots, (-1)^k (k-1)^k     \right).
\end{align*}
So, the equalities \eqref{lc5} hold.

\medskip
\noindent
For (i), since $\phi_k=\langle \mathbf{a}, V_k(x)\rangle$,   it is straightforward to verify that the conditions \eqref{eq2} for $\phi_k$ to be a model function are equivalent to the system $R_k\,\mathbf{a} =\mathbf{b}$. So, $U_k=\mathbf{a}$ must be a solution of this system. As $R_k$ is nonsingular, this is the unique solution. Furthermore, since $z$ is a simple zero of $f$, the equality  \eqref{lc8} holds because it is just the first equation of the system $R_k\,\mathbf{a} =\mathbf{b}$.

\medskip
\noindent
For (ii), we need to show that for $\mathbf{\alpha}=(\alpha_0, \alpha_1, \ldots, \alpha_k)\in \Rb^{k+1}$, such that
\begin{equation}\label{above}
\alpha_0 f^{(1)}(x)+ \alpha_1 f^{(1)}(x+h(x))+ \ldots+  \alpha_k f^{(k)}(x+k\, h(x))=\mathbf{0},
\end{equation}
the only solution is $ \mathbf{\alpha}=\mathbf{0}$. Differentiating  \eqref{above} and evaluating  at $x=z$, we obtain the homogeneous linear system
$$
\operatorname{diag}\left(f^{(1)} (z), f^{(2)} (z), \cdots, f^{(k+1)} (z)\right) R_k\, \mathbf{\alpha}= \mathbf{0},
$$
which admits only the   solution $\mathbf{\alpha}=\mathbf{0}$ since  both the diagonal matrix and $R_k$ are nonsingular.

\medskip
\noindent
The last assertion follows from Proposition~\ref{prop1} since by item (i) $\phi_k$ is a model function.
\end{proof}

\noindent
The expressions for the first five barycentric  maps are shown in Table \ref{tab2}.
    \begin{table}
$$
\hspace{-1.9cm}
\small
\begin{array}{ |   l  | }
\hline
t_1(x)= x- \displaystyle{ \frac{2\, f(x)}{f^{(1)}(x)+ f^{(1)}(x+h(x))}   }\\
t_2(x)=  x- \displaystyle{ \frac{12\, f(x)}{5\, f^{(1)}(x)+8  f^{(1)}(x+h(x))-  f^{(1)}(x+2\,h(x))}   }\\
t_3(x)= x- \displaystyle{ \frac{24\, f(x)}{9\, f^{(1)}(x)+19  f^{(1)}(x+ h(x))-5\,  f^{(1)}(x+ 2\,h(x))+ f^{(1)}(x+3\, h(x))}   }\\
t_4(x)= x- \displaystyle{ \frac{720\, f(x)}{251\, f^{(1)}(x)+646  f^{(1)}(x+h(x))-264\,  f^{(1)}(x+2\, h(x))+106\, f^{(1)}(x+3 h(x))-19  f^{(1)}(x+4 h(x))}   }\\
t_5(x)= x- \displaystyle{ \frac{1440\, f(x)}{\sum_{i=0}^{5} \alpha_1\, f^{(1)}(x+i\, h(x))}   }, \quad \mbox{with}\quad \alpha_0=475, \, \alpha_1=1427,\, \alpha_2=-798\\
\hspace{6cm}\begin{array}{l}
\alpha_3=482, \, \alpha_4=-173,\, \alpha_5=27\\
\end{array}\\
\hline
  \end{array}
  $$
  \caption{ First five barycentric type maps. \label{tab2}}
  \end{table}

\section{Recursive families of iterative maps}\label{secrec}

\noindent
We recall that a model function $\phi$ depends on a  certain step function $h$. Now, for each model function $\phi_j$ entering in the definition of the map $t_j=x-\left[ \phi_j(x) \right]^{-1} f(x)$, we use a step function which is defined recursively by $h_j(x)=t_{j-1}(x)-x$.  The  starter $t_0$ will be taken to be the Newton's map $t_0(x)=x-[f^{(1)}(x)]^{-1} f(x)$.
 The next proposition shows that the  iterative map $t_m$, defined recursively in \eqref{step1A},  has local order of convergence $m+2$.
\begin{proposition}\label{propD} Let $z$ be a simple zero of a given  function $f$ and   $t_0 $ the Newton's  map
$$t_0(x)=x-[f^{(1)}(x)]^{-1} f(x).$$
\medskip
\noindent
For a given natural number $m\geq 1$, define recursively the step function $h_m$ and the iterative map $t_m$ by
\begin{equation}\label{step1A}
\begin{array}{l}
h_j(x)= t_{j-1}(x)-x\\
\hspace{5cm}\quad j=1,2,\cdots,m\\
t_j(x)= x- \left[ \phi_j(x)   \right]^{-1}\, f(x),
\end{array}
\end{equation}
where $\phi_j $ is  constructed using $h_j$ as step function   and $\phi_j$ is either a Taylor\--type or a barycentric\--type map, respectively given by \eqref{eq10B} and \eqref{lc4}.  Then, the   map $t_m$ has local order  of convergence  at least $m+2$.
\end{proposition}
\begin{proof} It is only necessary to prove that each function $h_j$ is a step function and the statement follows from Propositions \ref{propB} and \ref{lemaC}. 

\medskip
\noindent
Let us apply   induction on the integer $m$.  
For $m=1$, we  have $h_1(x)=t_0(x)-x$ and so $h_1(z)=0$ and $h_1^{(1)} (z)=-1$.
\medskip
\noindent
Let $m\geq 1$ be an integer. As $h_m^{(0)}(z)=0$, $h_m^{(1)}(z)=-1$ and for any integer $i$ such that $2\leq i \leq m$, we have $h_m^{(i)}(z)=0$, and so $h_m$ is a  step function. 
\end{proof}

\medskip\noindent
We call {\it Newton\--Taylor} and {\it Newton-barycentric} maps those maps $t_k$, defined in the Proposition~\ref{propD}, when one  considers the model function to be respectively a Taylor-type map and a  barycentric\-type map. As before the name of these maps was chosen in order to emphasize that the starter step  function is the Newton map.

\bigskip
\noindent
{\bf Newton\--Taylor maps}\label{subnb}

\medskip
\noindent
Let us compute  the explicit expressions for the first two Newton\--Taylor maps described in the previous proposition. The order of convergence of the first map $t_1$ and of the second map $t_2$ is respectively 3 and 4.
 
 \medskip
\noindent
{\bf Order 3}:
$$
t_1(x)=x- \displaystyle{ \frac{2\, f(x)}{2\, f^{(1)}(x)+ f^{(2)} (x)\, \left(  \displaystyle{\frac{- f(x)}{f^{(1)}(x)}} \right)}  }= x- \displaystyle{  \frac{2\, f^{(1)} (x)\, f(x)}{2 \left( f^{(1)}(x) \right)^2-f(x)\, f^{(2)}(x)}  }.
$$
The map $t_1$ coincides with the celebrated Halley's method (see \cite{halley} and \cite{traub}).

\medskip
\noindent
{\bf Order 4}:
$$
t_2(x)= x-\displaystyle{
 \frac{6 f(x)\, (f^{(1)}(x))^2}{  f^2 (x)\, f^{(3)}(x)+ 6\, (f^{(1)}(x))^3- 3 f(x)\, f^{(1)}(x) f^{(2)}(x)}  .
       }
       $$
 
 \bigskip
 \noindent
{\bf Newton\--barycentric maps}\label{subnbA}

\medskip
\noindent
Let us compute the explicit expressions of the first two Newton-barycentric   maps $t_1$ and $t_2$.  Recall that $t_0(x)=x- [f^{(1)} (x)]^{-1}\, f(x)$ and 
\begin{equation}\label{step1}
h_1(x)= t_0(x)-x= -  \frac{f(x)}{f^{(1)} (x)}.
\end{equation}
 Then 
$$\phi_1(x)= a_0 f^{(1)} (x)+ a_1 f^{(1)} \left(x+h_1(x)\right),$$
where $\mathbf{a}= (a_0,a_1)$ is the solution of the linear system \eqref{lc5} with $k=1$. This solution is $\mathbf{a}= (1/2, 1/2)$ and so the Newton\--barycentric map of order 3 is given by

\medskip
\noindent
{\bf Order 3}:
\begin{equation}\label{order3}
t_1(x)=x- \displaystyle{ \frac{2\, f(x)}{f^{(1)}(x)+f^{(1)}(x+h_1(x))  }  },
\end{equation}
with $h_1$ is as in \eqref{step1}.

\bigskip
\noindent
For the map of order 4, we have 
$$\phi_2(x)= a_0 f^{(1)} (x)+ a_1 f^{(1)} \left(x+h_2(x)\right)+ a_2 f^{(1)} \left(x+2h_2(x)\right),
$$
with
$$h_2(x)= t_1(x)-x = \frac{-2 f(x)}{f^{(1)}(x)+f^{(1)}\left(x+h_1(x)\right)},$$
and $\mathbf{a}=(a_0,a_1,a_2)$  is the solution of the system \eqref{lc5} with $k=2$. This solution is $\mathbf{a}=\frac{1}{12}(5,8,-1)$,  which gives

\medskip
\noindent{\bf Order 4}:
\begin{equation}\label{order4}
t_2(x)=x- \displaystyle{
 \frac{12\, f(x)}{
  5\,f^{(1)}(x)+ 8\, f^{(1)} \left(  x-  \frac{2 \, f(x)}{f^{(1)}(x)+f^{(1)}(x+h_1(x))}           \right)-f^{(1)}\left( x-\frac{4\, f(x)}{f^{(1)}(x)+f^{(1)}(x+h_1(x)) }   \right)  
   } .
  }
\end{equation}
The next three Newton-barycentric formulas are given in Table \ref{tab2} where in each $t_k$ the step function $h$  should be substituted by $h_{k}$ given by \eqref{step1A}.

\bigskip
 
\noindent{\bf Extension of the Newton-barycentric maps to $\Rb^n$}\label{extend}

\medskip
\noindent
In the  numerical examples presented in next section we apply Newton\--barycentric maps  to functions $f$ defined in $\Rb^2$. The following modifications were made in order  to extend the Newton\--barycentric maps to functions  $f: \Rb^2\mapsto \Rb^2$. 

\medskip
\noindent
The derivative $f^{(1)}$ is substituted by the Jacobian operator, that is  $f^{(1)}(x)=\left[
\partial f_i (x)/\partial x_j 
 \right]_{i,j=1}^n$.
 Assuming that this  matrix is nonsingular and taking $\Delta t_k(x)= t_k(x)-x$, the respective model function $\phi_k$ verifies
 $$
 \phi_k(x)\, \Delta t_k(x)=-f(x).
 $$
Considering an initial guess $x^{(0)}\in \Rb^n$, each vector resulting from applying the map $t_k$ is computed by solving the linear system
 \begin{equation}\label{orderAA}
 \begin{array}{l}
 \phi_k (x^{(i)})\, \Delta t_k(x^{(i)})= -f (x^{(i)})\\
 x^{(i+1)}=x^{(i)}+ \Delta t_k( x^{(i)}  ), \quad i=0,1,\ldots.
 \end{array}
 \end{equation}

 \section{Numerical examples }\label{secexemp}
 
 In this section we apply the Newton-barycentric  maps in two numerical examples. In Example \ref{ex1} we locate a small number of extrema of a typical least squares problem and in Example \ref{ex2} we locate a great number of local extrema of a function related to Ackley's function (see \cite{ackley}), using the scheme \eqref{orderAA} with $i=0, 1$. 

\medskip
\noindent
In order to find the zeros of such functions let us start by detailing the procedure to be followed.
 For a given function $f$, defined in a rectangle $D=[x_{min}, x_{max}]\times [y_{min}, y_{max}]\subset \Rb^2$, we consider a rectangular grid  in  $D$ having mesh widths $d_x$ and $d_y$. We take for  data points  the  vertices of the grid which are stored in a  list  ${\cal L}$.  
 
 \medskip
 \noindent
 We recall that  an iterative map $t$  of order of convergence at least $2$ leads to a superlinear  iterative process and so,  for any initial point sufficiently close to a fixed point of $t$, the respective iterates either converge to the fixed point or go away from it.  Therefore, given a $k\geq 0$,  applying a Newton-barycentric map $t_k$ to each point in ${\cal L}$, the  image points are either attracted to, or repelled from,  the fixed points of $t_k$ eventually lying in $D$.

 \medskip
 \noindent
 In order to test numerically some of the Newton-barycentric maps we only  apply   two iterations to the data points in the list  ${\cal L}$. For this purpose, we consider $\epsilon$ to be a given tolerance and  denote by $X^{(0)}$ an element of  ${\cal L}$. The first and second iterates of a given map $t_k$ are denoted by  $X^{(1)}$  and  $X^{(2)}$ respectively. For each point $X^{(0)}$ in the list ${\cal L}$ we consider the following algorithm:
\begin{itemize}
\item[1.] If  the Jacobian matrix $J_f(X^{(0)})$ is singular,   the point $X^{(0)}$ is ignored and the next point in  ${\cal L}$ is assigned to  $X^{(0)}$. 
\item[2.] If both $X^{(1)}=t_k(X^{(0)}) $ and $X^{(2)}=t_k(X^{(1)})$ do not belong to the domain $D$,  the  point $X^{(0)}$ is ignored and  $X^{(0)}$ is taken to be the next point in  ${\cal L}$ and proceed to step 1. 
\item[3.] If $|| f(X^{(2)})||\leq \epsilon$ store $X^{(2)}$ in a list ${\cal C }$, otherwise let $X^{(0)}$ be the next point in ${\cal L}$ and go to step 1.
 \end{itemize}
 
 \medskip
 \noindent
 After testing all the elements in ${\cal L}$, if the  list  of the captured points ${\cal C }$  is not empty its elements will cluster near a fixed point of the iterative map $t_k$ in the search domain $D$. So, the plot of the captured points in ${\cal C}$ gives us a picture of the location of the fixed points of $t_k$. In fact, as it is well-known  (see for instance  \cite{dennis1}), for an iterative map of order $p\geq 2$, the error of an iterate $X^{(i)}$ is approximately $X^{(i+1)}-X^{(i)}$ and therefore it is expectable that the point $X^{(2)}$ will be closer to a fixed point of the map than $X^{(0)}$ and $X^{(1)}$. 
 
 \medskip 
 \noindent
 Obviously we are not claiming that only two iterations of a higher order map are sufficient to locate all the simple zeros of a function $f$ in a domain $D$,  by inspecting   the  list  of the captured points ${\cal C }$. For a given tolerance $\epsilon$, one can only say that the captured points are likely to be close approximations of the zeros of $f$ eventually lying in $D$. In particular,  for a discussion on the numerical validation of a few number of iterations of Newton's method the reader is referred to \cite{alefeld}, and  for the fundamental question of proving the existence of zeros of nonlinear maps in $\Rb^n$ see for instance \cite{frommer} and the references therein.

    \begin{figure}[h]
\begin{center} 
 \includegraphics[totalheight=4.2cm]{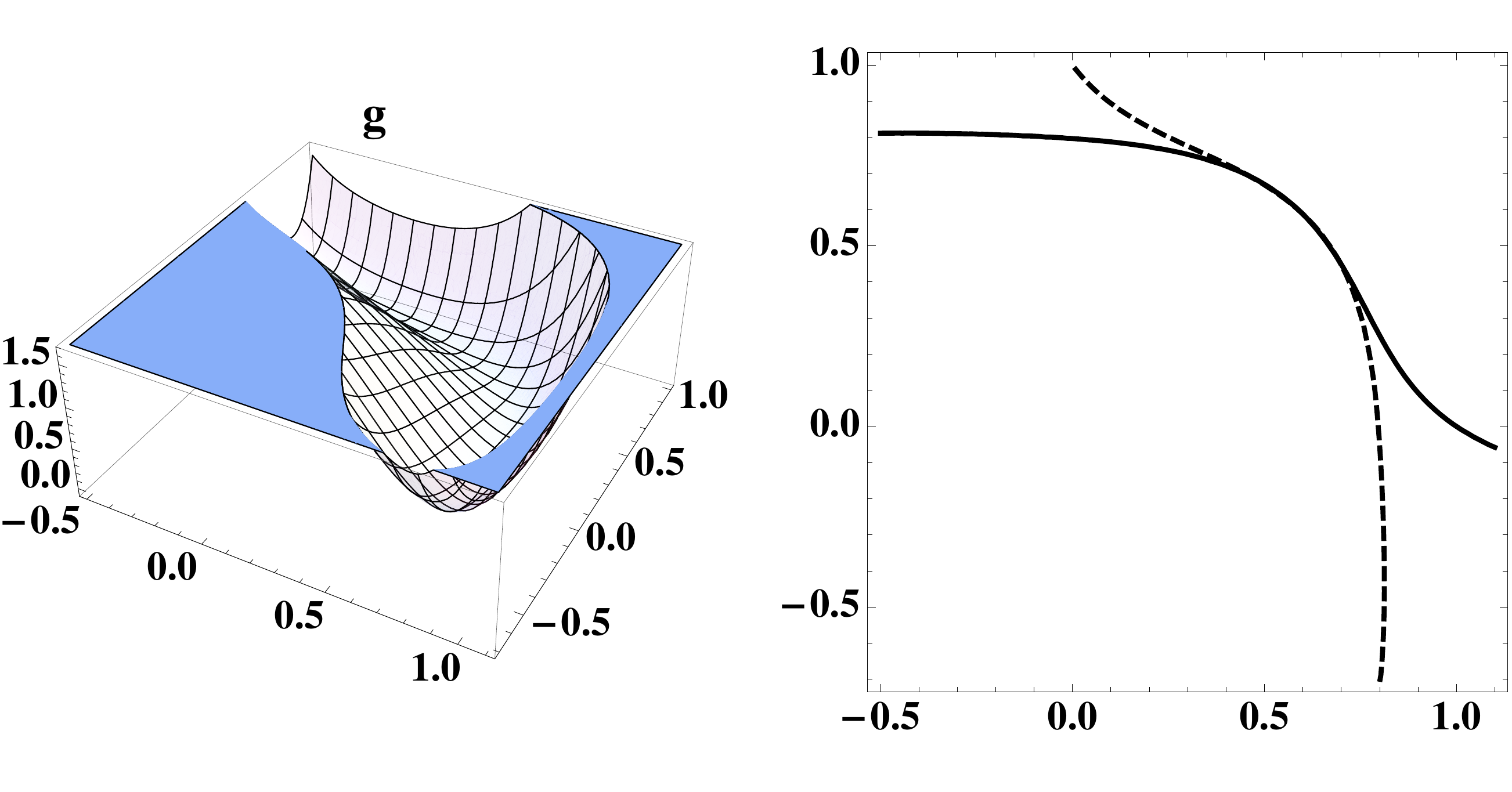}
\caption{\label{figrut1} The ``landscape"  of the function $g$ (left) and the zero level curves of the components of $f$ given in the Example \ref{ex1}.}
\end{center}
\end{figure}

    \begin{example}\label{ex1} Consider the system
    $$
    \left\{
    \begin{array}{l}
    x+y=1\\
    x^2+y^2=0.8\\
    x^3+y^3=0.68\\
     x^4+y^4=0.01\\ 
    \end{array}
    \right.
    $$
    The first three equations were considered by Rutishauser \cite{rutishauser}  for illustrating a least squares problem. A least squares solution for the system 
    can be found by  minimizing the function
    $$
    g(x,y)= s_1^2+s_2^2+s_3^3+s_4^2,
    $$
    where $s_1$ to $s_4$ are the residuals
    $s_1=x+y-1$, $s_2=x^2+y^2-0.8$, 
 $s_3=  x^3+y^3-0.68$, and $ s_4=x^4+y^4-0.01$.
 
 \medskip
\noindent
We apply some Newton-barycentric maps (see Table \ref{tab2}) in order to locate the roots of the equation $f(x,y)=\nabla g(x,y)=(0,0)$, where the components of the function $f=(f_1,f_2)$ are the following polynomials
$$
\begin{array}{l}
f_1(x,y)=2.-1.2 \,x-4.08\, x^2+3.92\, x^3+6 \,x^5+\\
\hspace{3cm} +8\, x^7+2\, y+4\, x \,y^2+6 x^2\, y^3+8\, x^3 \,y^4,\\
\\
f_2(x,y)=-2.+2 \,x-1.2\, y+4\, x^2 \,y-4.08\, y^2+6\, x^3 \,y^2+\\
\hspace{3cm} +3.92\, y^3+8 \,x^4 \,y^3+6 \,y^5+8\, y^7.
\end{array}
$$

\noindent
In Figure \ref{figrut1} we show  both  the plots of the of function $g$  ``landscape"  and  the zero level curves of $f_1$ and $f_2$. Since   $g$ has a flat  ``valley" and the zero level curves of $f$ seem to cross in several points it is not clear at all if a global minimum exists for this function in the following domain
\begin{equation}\label{Dex1}
D=[x_{min},x_{max}]\times [y_{min},y_{max}]=[-0.5,1.1]\times [-0.7,1.1].
\end{equation}
    \begin{figure}[h]
\begin{center} 
 \includegraphics[totalheight=8.0cm]{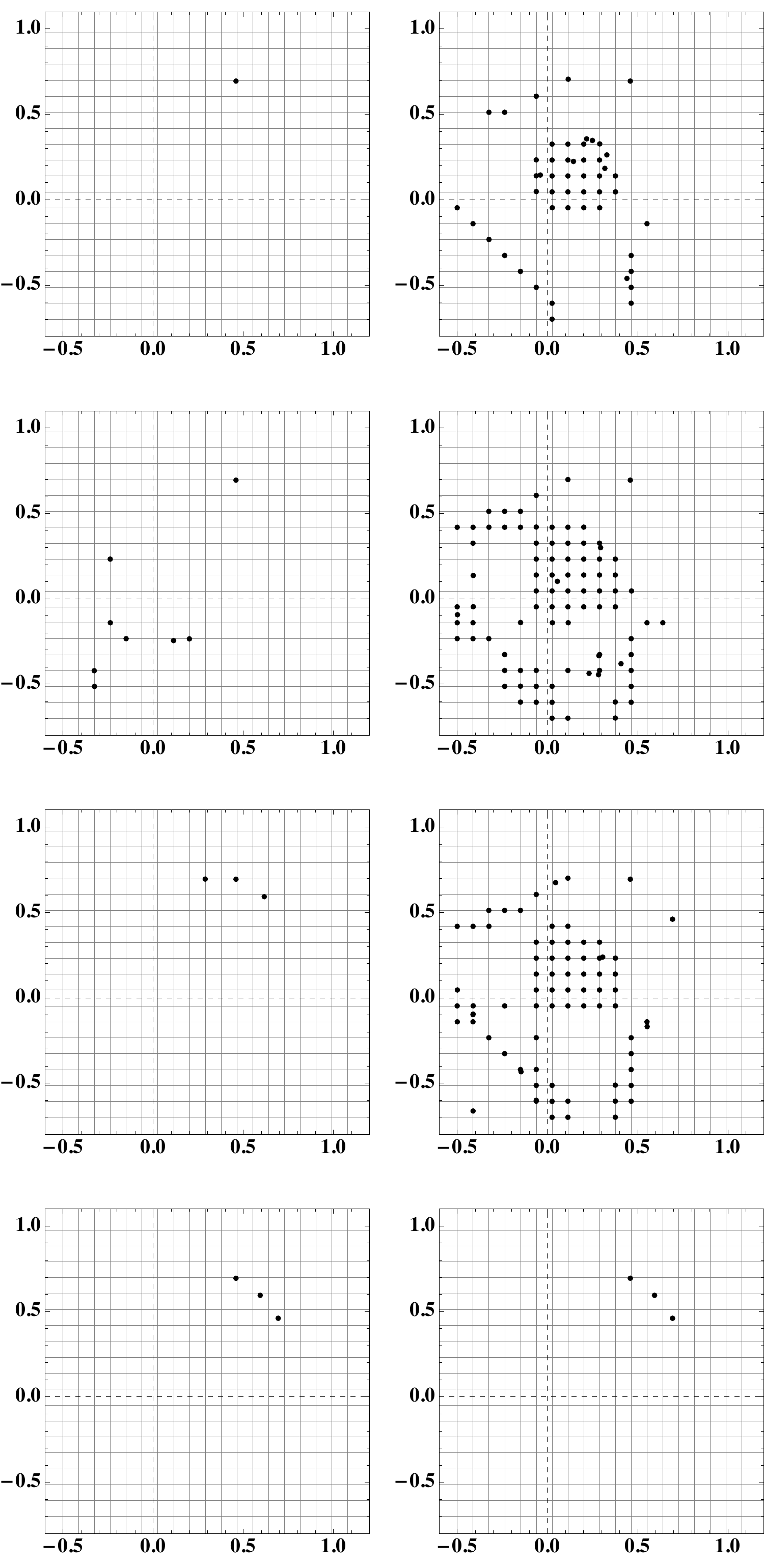}
\caption{\label{figrut2}  Captured points after two iterations of $t_0$, $t_2$, $t_4$, $t_{32}$ (left column) and $t_1$, $t_3$, $t_5$, $t_{43}$ (right column) \--- Example \ref{ex1}.}
\end{center}
\end{figure}

\noindent
We consider  a tolerance $\epsilon=0.001$ and  a rectangular grid having mesh widths respectively $d_x\simeq 0.0876712$ and $d_y\simeq 0.0931507$. 
 The respective list ${\cal L}$ contains the vertices of the mesh, that is  $N=19\times 19=361$ data points,  belonging to the search domain $D$. The previously described algorithm is applied  to the data points using respectively the Newton\--barycentric maps $t_0$ to $ t_5$ and $t_{21}, t_{32}$ (the map $t_{ij}$ is the composition $t_i(t_j)$).   The computations were carried out using the system {\sl Mathematica} \cite{wolfram1} in double precision.

\medskip
\noindent
The elements of the list of captured points (after two iterations of each map) are shown in Figure \ref{figrut2}.  The number of captured points is given in Table~\ref{tabg3}  and the captured points and their respective value by $g$ (rounded to $6$ decimal places) are in Table \ref{tabg4}.  From Figure~\ref{figrut2}, we see that the  $18$ captured points  cluster near 3 distinct points in the search domain  $D$. It is clear from Table \ref{tabg4} that two  global minimum have been located, one of them  at  $P=(0.459591,0.693716)$, with $g(P)=0.167974$.

\begin{table}
$$
\begin{array}{| c| c| c| c| c| c| c| c|c|}
\hline
t_i& t_0&t_1&t_2&t_3&t_4&t_{5}&t_{21}&t_{32}\\
\hline
\verb+#+& 1   &50   &  8 & 89 &4 & 77&6&18  \\
 \hline
  \end{array}
  $$
  \caption{Number of captured points in the domain $D$ given in \eqref{Dex1} \--- Example \ref{ex1}. \label{tabg3}}
  \end{table} 
 
 \medskip 
  
    \begin{figure}[h]
\begin{center} 
 \includegraphics[totalheight=6.7cm]{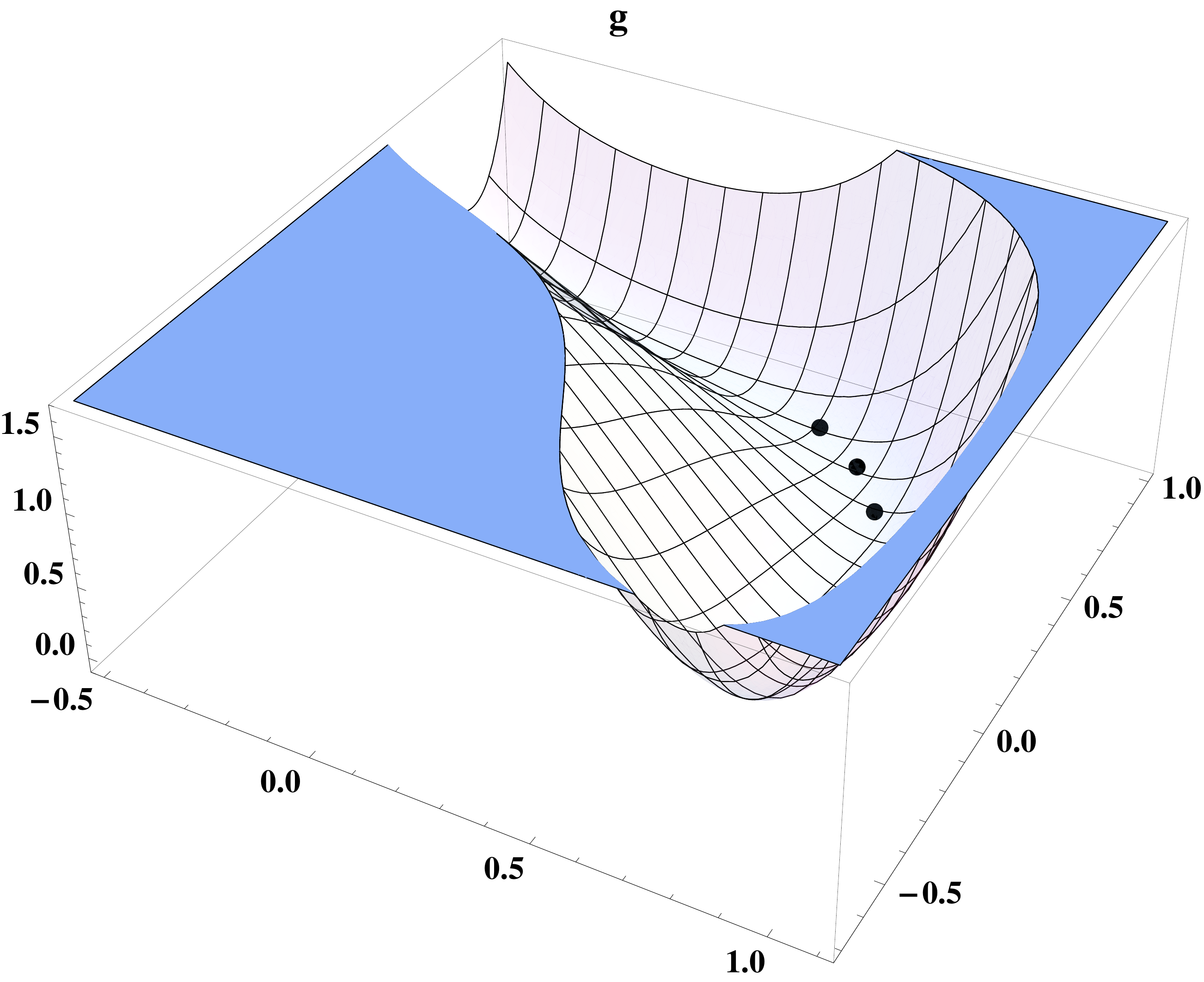}
\caption{\label{figrut3}  The three bold black points represent the clusters of captured points obtained after two iterations of $t_{32}$ \--- Example \ref{ex1}.}
\end{center}
\end{figure}

 \begin{table}
$$
\begin{array}{| c  | }
\hline
(x_i,y_i, g(x_i,y_i))\\
\hline
(0.459591,0.693716,0.167974)\\
\hline
(0.693716,0.459591,0.167974)\\
\hline
(0.593976,0.593976,0.169389)\\
\hline
   \end{array}
  $$
  \caption{Three of the 18 points   captured by $t_{32}$ (the other  coincide up to 6 decimal places) with tolerance $\epsilon=0.001$\---  Example \ref{ex1}.\label{tabg4}}
  \end{table} 
 
 \medskip
 \noindent
 \end{example}
      \begin{figure}[h]
\begin{center} 
 \includegraphics[totalheight=4.0cm]{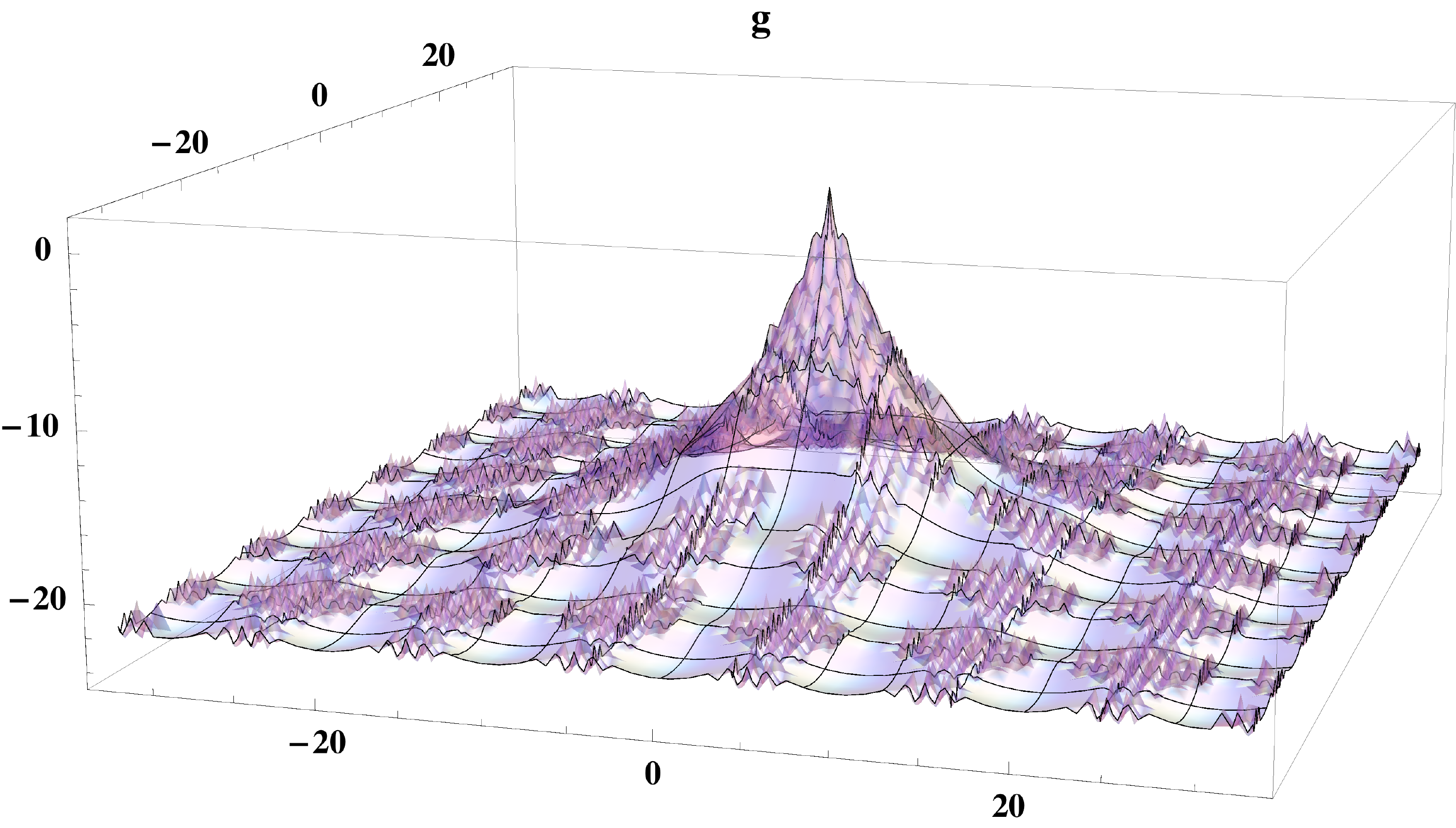}
\caption{\label{figg} A 3D plot of the function $g$ given in \eqref{g1}.}
\end{center}
\end{figure}

  
 \noindent
 
     \begin{example}\label{ex2}
      In this example   two iterations of Newton-barycentric  maps  are used in order to  locate simultaneously a great number of zeros of a  function $f$, related to the famous Ackley's function \cite{ackley}.   
 We consider the following function 
     \begin{equation}\label{g1}
     \begin{array}{l}
     g(x,y)= - (-20 \,e^{s_1}- e^{s_2}+20+e), \quad \mbox{where}\\
     \\
     s_1=- 0.2 \sqrt{0.5\, (x^2+y^2)}\quad \mbox{and}\quad  s_2= 0.5\left( \cos(2\pi\, x)+ cos(2\pi\,y) \right).
     \end{array}
\end{equation}
The function $g$ is the symmetric of the  Ackley's function which is widely used  for testing optimization algorithms (see for instance \cite{more}).

     \medskip
     \noindent
 We consider the standard search domain $D=[x_{min}, x_{max}]\times [y_{min}, y_{max}]\subset \Rb^2$, with
     $$
    x_{min}=y_{min}= -32.768\quad \mbox{ and} \quad  x_{max}=y_{max}= 32.768.
     $$
     The ``landscape" of $g$ is shown in Figure \ref{figg}. 
     At $(0,0)$ the function $g$ has a global maximum equal to zero and a great number of local extrema in $D$. Some extrema  will be (simultaneously) located  by applying two iterations of  Newton-barycentric maps, namely  \eqref{order3}, \eqref{order4}. 
     For that purpose we consider  $f:D\subset\Rb^2\mapsto \Rb^2$ to be the gradient of $g$ and we look for its zeros.
            \begin{table}
$$
\begin{array}{| c| c| c| c|}
\hline
x_i&y_i&f(x_i,y_i)& g(x_i,y_i)\\
\hline
-1.65185& -1.65185& (1.33227*10^{-15}, 1.33227*10^{-15})& -7.7843\\
   \hline
    -1.65185&  1.65185& (1.33227*10^{-15}, -1.33227*10^{-15})& -7.7843 \\
    \hline
-1.65185&-1.65185&(1.33227*10^{-15}, 1.33227*10^{-15})&  -7.7843\\
    \hline
 -1.6103& 0.&(-1.33227*10^{-15}, 0.)& -5.66925\\
     \hline
   -1.65185&  1.65185& (1.33227*10^{-15}, -1.33227*10^{-15})&-7.7843\\    
       \hline
   0.& -1.6103& (0., -1.33227*10^{-15})& -5.66925\\
       \hline
 0.& 1.6103&(0., 1.33227*10^-{15})& -5.66925\\
     \hline
    1.65185& -1.65185& (-1.33227*10^{-15}, 1.33227*10^-{15})& -7.7843\\
        \hline
    1.6103&  0.&(1.33227*10^{-15}, 0.) & -5.66925\\
        \hline
   1.65185&  1.65185& (-1.33227*10^{-15}, -1.33227*10^{-15})& -7.7843\\
       \hline
   1.65185&-1.65185& (-1.33227*10^{-15}, 1.33227*10^{-15})& -7.7843\\
       \hline
 1.65185& 1.65185&(-1.33227*10^{-15}, -1.33227*10^{-15})& -7.7843\\
     \hline
   \end{array}
  $$
  \caption{Tolerance $\epsilon=0.1$ and a square mesh of width $d\simeq 1.6384$. Captured points by $t_{54}$ which are  located at a distance not greater than 3 from the origin \label{tabg5}  \--- Example \ref{ex2}.}
  \end{table} 
 
 \medskip
 \noindent
  The components of $f=\nabla g= (f_1,f_2)$ are
$$\begin{array}{ll}
f_1(x,y)=&-\displaystyle{     \frac{2.8284271247461907\times e^{-0.14142135623730953\, \sqrt{x^2 + y^2}} \,x}{\sqrt{x^2 + y^2}}} - \\
& \hspace{0.3cm} - 3.141592653589793\times  e^{0.5\, (\cos(2 \,\pi\, x) +\cos(2\, \pi\, y))}\,
   \sin(2 \pi\,x)=0\\
   \\
   f_2(x,y)&=-\displaystyle{     \frac{2.8284271247461907\times e^{-0.14142135623730953\, \sqrt{x^2 + y^2}} \,y}{\sqrt{x^2 + y^2}}} - \\
& \hspace{0.3cm} - 3.141592653589793\times  e^{0.5\, (\cos(2 \,\pi\, x) +\cos(2\, \pi\, y))}\,
   \sin(2 \pi\,y)=0.
   \end{array}$$
The function $f$ is not defined at $x=y=0$, but can be continuously extended to the origin by taking $f(0,0)=(0,0)$.  Moreover,  the function is not differentiable at $(0,0)$ which explains  why  global search algorithms   can hardly find the global extremum of $f$  located at the origin.

\begin{figure}[h]
\begin{center} 
 \includegraphics[totalheight=9.2cm]{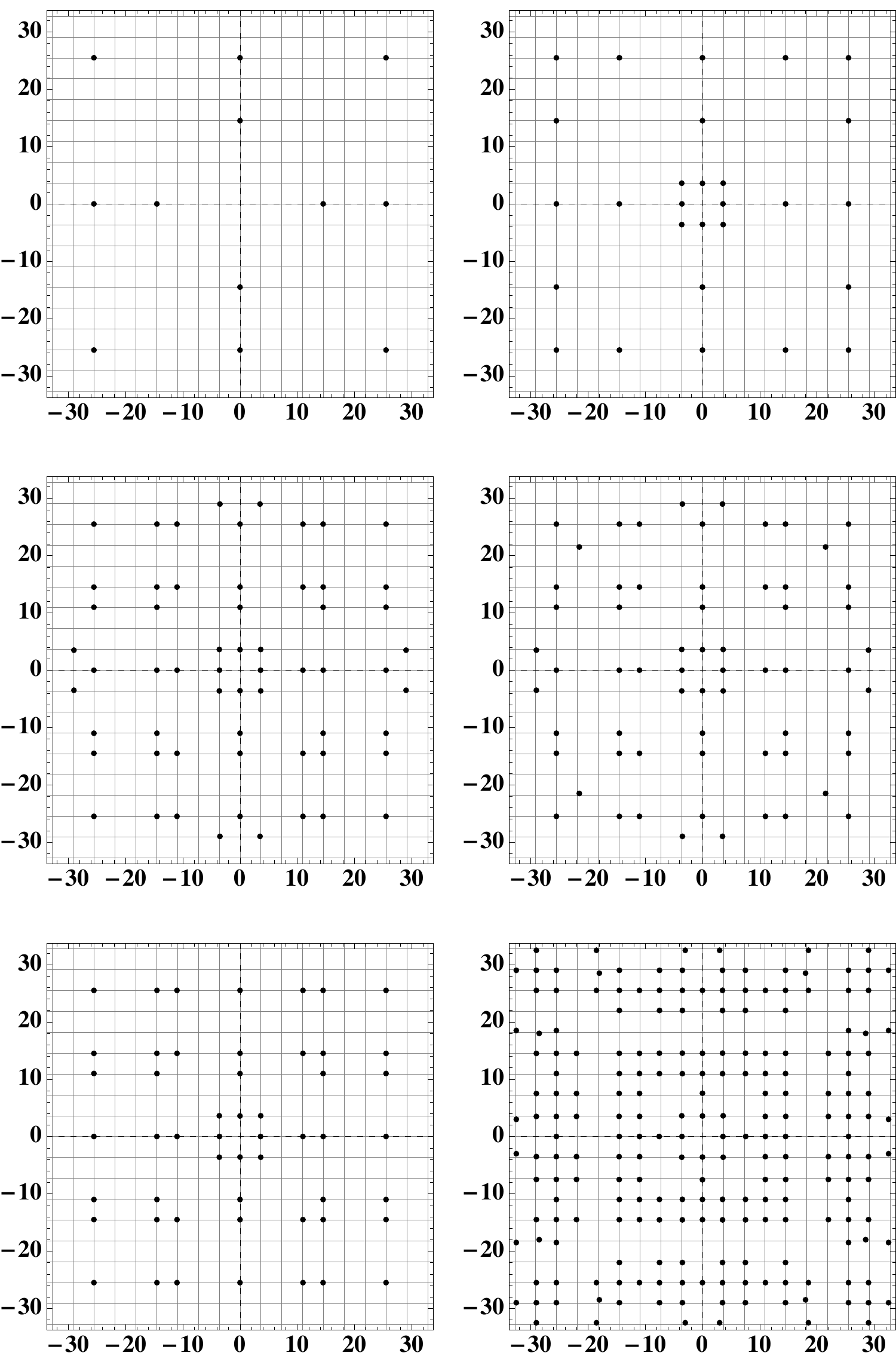}
\caption{\label{figg2} Tolerance $\epsilon=0.001$. Captured points after two iterations for $t_0$, $t_2$ and $t_4$ (left column), and $t_1$, $t_3$ and $t_{54}$ (right column) \--- Example \ref{ex2}.}
\end{center}
\end{figure}

   \begin{table}
$$
\begin{array}{| c| c| c| c| c| c| c|}
\hline
t_i& t_0&t_1&t_2&t_3&t_4&t_{54}\\
\hline
\verb+#+& 12  &28   &  60 & 64 &52 & 208   \\
 \hline
  \end{array}
  $$
  \caption{Number of captured points \--- Example \ref{ex2}.\label{tabg2}}
  \end{table} 

 \begin{figure}[h]
\begin{center} 
 \includegraphics[totalheight=4.3cm]{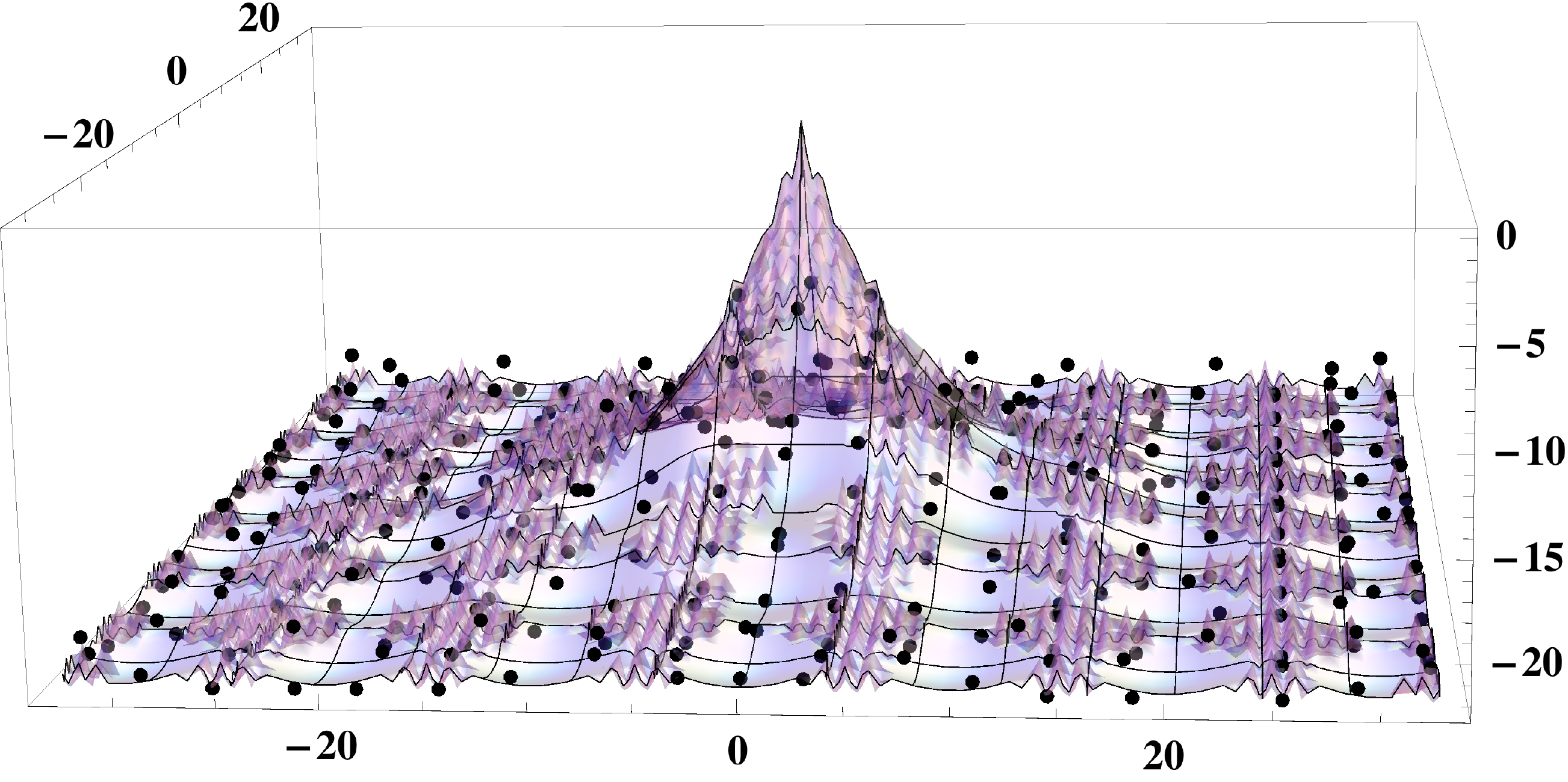}
\caption{\label{figg5} Tolerance $\epsilon=0.1$ and $d\simeq 1.6384$. The map  $t_{54}$ captures $1458$ points  after two iterations  \-- Example \ref{ex2}.}
\end{center}
\end{figure}
 
 \medskip
\noindent
For a tolerance $\epsilon=0.001$ we  consider the set of $N=361$ points in  $D$ formed by  the vertices of a uniform  large mesh of width $d=d_x=d_y\simeq 3.64089$. We apply two iterations of the Newton-barycentric maps $t_0$, $t_1$, $t_2$, $t_3$ and $t_{54}$ (the map $t_{54}$ is the composition $t_5\circ t_4$) to the data points following the algorithm previously described. 

\medskip
\noindent
The captured points are shown in Figure \ref{figg2} and the number of captured points  is given in Table \ref{tabg2}.   The map $t_{54}$  is used just for comparison purposes with   other Newton-barycentric maps of lower order.  From Figure \ref{figg2} it is reasonable to conclude that the zeros of $f$  are symmetrically located  with respect to the axes.  We remark that the Newton's map $t_0$ only captures $9$ points after two iterations and in this sense the other higher order methods might give more insight  on the location of the zeros of the function $f$ (or the extrema of Ackley's function).

\medskip
\noindent
For the tolerance $\epsilon=0.1$ and $N=1681$ data points obtained from a mesh of width $d\simeq 1.6384$, the map $t_{54}$ is able to capture $664$ points  $(x_i,y_i)$ clustering near the   zeros of the function $f$. In Figure \ref{figg5}  we present the 3D  plot of the captured points $(x_i,y_i,g(x_i,y_i))$ for  the function $g$ given in \eqref{g1}.  In order to observe how close are the captured points by $t_{54}$ to the zeros of $f$, we give in
 Table~\ref{tabg5}  the coordinates of the points which are at a distance from $(0,0)$ not greater than $3$ as well as the respective values for $f$ and $g$.
    \end{example}

\end{document}